\documentclass[reqno, 12pt]{amsart}

\usepackage{amsmath,amsthm,amssymb,amsfonts, amscd, bm, graphicx, mathrsfs}

\usepackage[cmtip,all]{xy}

\allowdisplaybreaks[1]

\usepackage{mathrsfs}

\usepackage[pagebackref=true, colorlinks]{hyperref}

\hypersetup{colorlinks=true,citecolor=blue,linkcolor=blue,urlcolor=blue,
pdfstartview=FitH}
\newtheorem{thm}{Theorem}[section]
\newtheorem{lem}[thm]{Lemma}
\newtheorem{prop}[thm]{Proposition}

\theoremstyle{definition}
\newtheorem{defn}[thm]{Definition}

\newtheorem{example}[thm]{Example}

\theoremstyle{remark}

\theoremstyle{plain}                    

\numberwithin{equation}{section}

\def\DJ{{\fontencoding{T1}\selectfont\char208}}

\newcommand{\bfz}{{\bf 0}}

\makeatletter
\DeclareRobustCommand\widecheck[1]{{\mathpalette\@widecheck{#1}}}
\def\@widecheck#1#2{%
    \setbox\z@\hbox{\m@th$#1#2$}%
    \setbox\tw@\hbox{\m@th$#1%
       \widehat{%
          \vrule\@width\z@\@height\ht\z@
          \vrule\@height\z@\@width\wd\z@}$}%
    \dp\tw@-\ht\z@
    \@tempdima\ht\z@ \advance\@tempdima2\ht\tw@ \divide\@tempdima\thr@@
    \setbox\tw@\hbox{%
       \raise\@tempdima\hbox{\scalebox{1}[-1]{\lower\@tempdima\box
\tw@}}}%
    {\ooalign{\box\tw@ \cr \box\z@}}}
\makeatother

\begin{document}

\title[A result on the size of iterated sumsets in $\mathbb{Z}^d$]{A result on the size of iterated sumsets in $\mathbb{Z}^d$}


\author[Ilija Vre\' cica]{Ilija Vre\' cica}

\address{
\begin{flushleft}
University of Belgrade,
Faculty of Mathematics, \\
Studentski Trg 16, p.p. 550,
11000 Belgrade, Serbia
\end{flushleft}
}

\email{ilijav@matf.bg.ac.rs}

\begin{abstract}
     In this paper we give a different approach to determining the cardinality of $h$-fold sumsets $hA$
     when $A\subset \mathbb{Z}^d$ has $d+2$ elements. This enables us to provide more
     general result with a shorter and simpler proof.

   We also obtain an upper bound for
     the value of $|hA|$ when $A\subset \mathbb{Z}^d$ is a set of $d+3$ elements with simplicial hull.
\end{abstract}


\date{\today}



\thanks{This work was  partially supported by  Ministry of Education, Science and Technological Development of  Republic of Serbia, Project no.  174034}

\maketitle


\section{Introduction}

An important area of study in arithmetic combinatorics is the $h$-fold sumset. For a set $A\subset\mathbb{Z}^d$, the $h$-fold sumset is

$$hA=\{a_1+\dots+a_h\, :\, a_i\in A\}.$$
An important contribution to this area of study is the following theorem due to Khovanskii:

\begin{thm}\cite{K}
    Given a finite set $A\subset\mathbb{Z}^d$, there exists a polynomial $p\in \mathbb{Q}[x]$ of degree at most $d$ such that $|hA|=p(h)$ for all sufficiently large $h$. Further, if $A-A$ generates all of $\mathbb{Z}^d$ additively, then $\deg p=d$ and the leading coefficient of $p$ is the volume of the convex hull of $A$.
\end{thm}

The proof of this theorem was, however, ineffective, as it yielded no information about the polynomial past its degree and leading term. There have been successes in bounding the integer $h_0$ such that $|hA|=p(h)$ for $h\geqslant h_0$ (\cite{GSW}, for instance).

However, in a recent paper (\cite{CG}), the cardinality of $|hA|$ has been completely described for all positive integers $h$, where $A\subset \mathbb{Z}^d$ is a set with $d+2$ elements, such that $A-A$ additively generates $\mathbb{Z}^d$.

\begin{defn}
    For $A\subset \mathbb{Z}^d$, we will denote the convex hull of $A$ with $\Delta_A$.
\end{defn}

The main result of the paper \cite{CG} is

\begin{thm}\label{thm:1_2_u_CG}\cite[Theorem 1.2]{CG}
    Suppose $A\subset \mathbb{Z}^d$ consists of $d+2$ elements, and further that $A-A$ generates $\mathbb{Z}^d$ additively. Then

    $$|hA|={h+d+1\choose d+1}\,\,\,\, {\rm whenever}\, 0\leqslant h< {\rm vol}(\Delta_A)\cdot d!$$
    and

    $$|hA|={h+d+1\choose d+1}-{h-{\rm vol}(\Delta_A)\cdot d!+d+1\choose d+1}\,\,\,\, {\rm whenever}\, h\geqslant {\rm vol}(\Delta_A)\cdot d!.$$
\end{thm}

The proof in \cite{CG} treats two cases in two different ways. The
first case is when $\Delta_A$ is a simplex with $d+1$ vertices and
$(d+2)$-nd element of $A$ is in $\Delta_A$, and the second case is
when $\Delta_A$ is a polytope with $d+2$ vertices. However, the
proof of the second case contained a misstep as noticed in the
first version of this paper, (see section 3). In the second
version of the paper \cite{CG} this proof is corrected using the
same idea as in the original proof.

In this paper we provide a different approach to the proof of
\cite[Theorem 1.2]{CG} and establish the more general result
treating also the sets $A$ for which the set $A-A$ does not
necessarily generate $\mathbb{Z}^d$ additively. This is our main
result and we show that the theorem from \cite{CG} is its direct
corollary. An additional advantage of our approach is that we
obtain a shorter and simpler proof which does not treat two cases in
different ways but provides a unified proof.

In the last section we briefly discuss the sets $A\subset
\mathbb{Z}^d$ of $d+3$ elements. We want to show that this
situation is considerably more complicated, as two such sets with
the same convex hull and the same $d+2$ elements could produce
different polynomials. So, it might be of some interest to obtain
an upper bound for $hA$ in this case.

\section{Lemmas}

For $v=(v_1,\dots,v_d)\in \mathbb{Z}^d$, we define its lift to be $\widetilde{v}=(v_1,\dots,v_d,1)\in \mathbb{Z}^{d+1}$. If $v=(v_1,\dots,v_d)\in \mathbb{Z}^d$ and $h\in\mathbb{N}$, then we write $(v,h)$ instead of $(v_1,\dots,v_d,h)$, and refer to $h$ as the height of this point.

\begin{defn}
    For a set $A=\{v_1,\dots,v_k\}\subset \mathbb{Z}^d$, the cone of $A$ is

    $$\mathcal{C}_A:={\rm span}_\mathbb{N}(\widetilde{v}_1,\dots,\widetilde{v}_k)=\{n_1\widetilde{v}_1+\dots+n_k\widetilde{v}_k\,\,:\,\, n_1,\dots,n_k\in\mathbb{N}\}.$$
\end{defn}
To the cone $\mathcal{C}_A$, we associate the generating series

$$\mathcal{C}_A(t):=\sum_{a\in\mathcal{C}_A}t^{{\rm height}(a)}.$$

Since the points at height $h$ form a copy of $|hA|$ embedded in $\mathbb{Z}^{d+1}$, we have that

$$\mathcal{C}(t)=\sum_{h\geqslant 0}|hA|t^h.$$

Let $A$ be a $d+3$ element set in $\mathbb{Z}^d$ such that $A-A$ generates $\mathbb{Z}^d$ additively and $\Delta_A$ is a $d$-simplex. Denote the $d+1$ vertices of $\Delta_A$ by $v_1,\dots,v_{d+1}$. These span a lattice $\Lambda={\rm span}_{\mathbb{Z}}(\widetilde{v}_1,\dots,\widetilde{v}_{d+1})$ in $\mathbb{Z}^{d+1}$. For such a lattice, we denote the set

$$\Pi :=\left\{ \sum_{i=1}^{d+1}\lambda_i\widetilde{v}_i\,:\, 0\leqslant \lambda_i<1\right\}\cap \mathbb{Z}^{d+1}.$$

In the paper, we will also encounter $\Lambda^+:={\rm span}_{\mathbb{N}}(\widetilde{v}_1,\dots,\widetilde{v}_{d+1})$, and define $N_\Lambda$ as the cardinality of $\Pi$.

We partition $\mathcal{C}_A$ into residue classes $\pi$ (mod $\Lambda$), each of which can be represented by an element of $\Pi$. For $\pi \in \Pi$, we denote its residue class by $\mathcal{S}_\pi$. An element $(g,N)\in \mathcal{S}_\pi$ is said to be minimal if $(g,N)-\widetilde{v}_i$ does not lie in $\mathcal{C}_A$ for any $i$.

From the geometry of numbers, we know that if we have a lattice \\$\Lambda={\rm span}(\widetilde{v}_1,\dots,\widetilde{v}_{d+1})$ in $\mathbb{Z}^{d+1}$ with a fundamental domain of nonzero volume, then $\mathbb{Z}^{d+1}/\Lambda$ can be identified with the set of lattice points in the fundamental domain of $\Lambda$, and that this number is equal to the determinant of the matrix whose columns are the generating vectors $\widetilde{v}_i$. Therefore,

$$N_\Lambda=|\mathbb{Z}^{d+1}/\Lambda|={\rm vol}(\Delta_A)d!.$$
(For this claim, we refer the reader for example to \cite[Ch. 6, Sec. 1]{N}.)

\begin{lem}\cite[Lemma 3.1]{CG}\label{lem:broj klasa}
    If $(\alpha,M)$ is a minimal element of $\mathcal{S}_\pi$, then

    $$M\leqslant N_\Lambda -1.$$
\end{lem}

\begin{lem}\label{lem:racionalnost}
    Let $v_1,\dots, v_{d+1}\in \mathbb{Z}^d$ be vectors that generate $\mathbb{Z}^d$ and $\det (\widetilde{v}_1,\dots ,\widetilde{v}_{d+1})\\ \neq 0$, and let $w\in \mathbb{Z}^d$ be an integer vector such that $\widetilde{w}=a_1\widetilde{v}_1+\dots+a_{d+1}\widetilde{v}_{d+1}$, where $a_i$ are non-negative coefficients such that $a_1+\dots +a_{d+1}=1$. Then $a_i$ are rational numbers.
\end{lem}

\begin{proof}
Since $\widetilde{v}_1,\dots,\widetilde{v}_{d+1},\widetilde{w}$ are all integer vectors and $\det (\widetilde{v}_1,\dots ,\widetilde{v}_{d+1})\neq 0$, by Cramer's rule coefficients $a_i$ will be rational numbers.
\end{proof}

We will also need a well known result from Combinatorial geometry,
Radon theorem and its extension which we prove here.

\begin{thm}\label{radon}
Every set $S$ of $d+2$ points in $\mathbb{R}^d$, could be split in
two disjoint subsets $S=S_1\cup S_2$ such that the convex hulls of
$S_1$ and $S_2$ intersect, i.e. $\text{\rm conv} S_1\cap \text{\rm
conv} S_2\neq \emptyset .$
\end{thm}

Actually we need the following extension of this theorem, saying
that in generic case (when no $d+1$ points belong to the same
hyperplane), this splitting is unique.

\begin{thm}\label{unique}
Let $S=\{x_1,x_2,...,x_{d+2}\}$ be the set of points in
$\mathbb{R}^d$ (no $d+1$ of which belong to the same hyperplane)
and let $S=S_1\cup S_2$  be the splitting satisfying $\text{\rm
conv} S_1\cap \text{\rm conv} S_2\neq \emptyset .$ Then two points
$x_i,x_j\in S$ belong to the same of two sets $S_1$ and $S_2$ if
and only if they belong to different halfspaces determined by the
hyperplane spanned by the remaining $d$ points of the set $S.$
\end{thm}

\begin{proof}
Let the points $x_i,x_j\in S$ belong to the same halfspace $H_+$
determined by the hyperplane $H$ spanned by the remaining $d$
points of $S$. Then points $x_i$ and $x_j$ could not belong to the
same of two sets $S_1$ and $S_2$. Namely, if $x_i,x_j\in S_1$,
then $S_2\subseteq H$ and $S_1\subseteq H_+$. Then $S_1\cap S_2
\subseteq H$ and so $\text{\rm conv} (S_1\setminus
\{x_i,x_j\})\cap \text{\rm conv} S_2\neq \emptyset .$ This is
impossible since the set $S\setminus \{x_i,x_j\}$ consists of $d$
points in generic position.

Let the points $x_i,x_j\in S$ belong to different halfspaces
determined by the hyperplane $H$ spanned by the remaining $d$
points of $S$. (For example, let $x_i\in H_+$ and $x_j\in H_-$.)
Then points $x_i$ and $x_j$ could not belong to different of two
sets $S_1$ and $S_2$. Namely, if $x_i\in S_1$ and $x_j\in S_2$,
then $S_1\subseteq H_+$ and $S_2\subseteq H_-$. Then $S_1\cap S_2
\subseteq H$ and so $\text{\rm conv} (S_1\setminus \{x_i\})\cap
\text{\rm conv} (S_2\setminus \{x_j\})\neq \emptyset .$ This is
impossible since the set $S\setminus \{x_i,x_j\}$ consists of $d$
points in generic position.
\end{proof}

\section{Addressing a misstep in \ref{thm:1_2_u_CG}, non-simplicial case}



    The approach in \cite{CG} in the non-simplicial case is based on presenting
    $\Delta_A$ as the union of two simplices with a common $d-1$ face. But, some
    convex polytopes with $d+2$ vertices cannot be split into two simplices. To see this,
    we will need the extension of Radon's theorem (see Theorem \ref{unique}): a set $X$
    of $d+2$ elements
    in $\mathbb{R}^d$ can be split into two disjoint subsets $X=X_1\cup X_2$ such that
    the convex hulls of $X_1$ and $X_2$ intersect and this splitting is unique
    (in a generic case). Namely, we proved that two vertices are in the same set
    ($X_1$ or $X_2$) if and only if they belong to different half-spaces
    determined by the hyperplane spanned by the remaining $d$ vertices.

    Let $A$ be a set with $d+2$ elements that has a convex hull that is split into two
    $d$-simplices with a common $d-1$ face. We will denote the set of vertices determined
    by the common face of
    these two simplices by $X_1$, and set $X_2$ will be comprised of the remaining two
    vertices. Then, the convex hulls of $X_1$ and $X_2$ will intersect. Therefore, in
    this situation, one set will always have $d$ elements, and one will have $2$
    elements. If a set may be split into two sets which both have at least three elements,
    and their convex hulls intersect, then the convex hull $\Delta_A$ cannot be split into
    two simplices.

    An easy example of such a set is $A=\{P_1(1,0,0,0),P_2(0,1,0,0),\\P_3(0,0,1,0),
    Q_1(0,0,0,1),Q_2(0,0,0,0),Q_3(1,1,1,-1)\}$. We see that convex hulls of
    $X_1=\{P_1,P_2,P_3\}$ and $X_2=\{Q_1,Q_2,Q_3\}$ intersect at common barycenter
    $\left(\frac{1}{3},\frac{1}{3},\frac{1}{3},0\right)$. It is easy to see that $A-A$
    generates $\mathbb{Z}^d$. If $A$ did have a splitting into two simplices, one set
    would have $2$, and one would have $4$ elements, which is not the case here. Since
    these splittings are unique, we conclude that $\Delta_A$ does not split into two simplices.


\section{The main theorem}

We start by an example.

\begin{example}\label{ex}

Let $d=2$ and consider the set of four points
$A=\{(1,0),(0,1),\\ (-1,0),(0,-1)\}$. For any $h\in \mathbb{N}$ the
set $hA$ contains one point with the first coordinate equal to $h$
or $-h$, two points with the first coordinate equal to $h-1$ or
$-(h-1)$, three points with the first coordinate equal to $h-2$ or
$-(h-2)$, etc. Finally, the set $hA$ contains $h+1$ points with
the first coordinate equal to $0$.

Therefore the cardinality of the set $hA$ is $|hA|=2(1+2+\cdots
+h)+h+1=(h+1)^2$.

Applying Theorem \ref{thm:1_2_u_CG} to the set $A$ would imply
that the cardinality of $hA$ is ${h+3 \choose 3}- {h-1 \choose
3}=2h^2+2$, which is incorrect. This happens since the set $A$
does not satisfy the assumption of Theorem \ref{thm:1_2_u_CG} that
the set $A-A$ generates $\mathbb{Z}^d$ additively.

\end{example}

Here we provide a little bit more general result which contains
Theorem \ref{thm:1_2_u_CG} as a special case, and also treats the
sets of points not satisfying the assumption that $A-A$ generates
$\mathbb{Z}^d$ additively, like in the example above.
\smallskip

Let us consider the set $A=\{v_1,...,v_{d+2}\}$ of $d+2$ points in
$\mathbb{Z}^d$, no $d+1$ of which are contained in the same
hyperplane. By $\widetilde{v}_i=(v_i,1)$ we denoted the lifts of
these points in $\mathbb{Z}^{d+1}$, for $i\in \{1,2,...,d+2\}$.

Let us now denote, for $i\in \{1,2,...,d+2\}$,
$D_i=\det(\widetilde{v}_1,\dots,\widetilde{v}_{i-1},\widetilde{v}_{i+1},\dots,\widetilde{v}_{d+2})$,
and $D=\mbox{GCD}(D_1,\dots,D_{d+2})$. We state now our main
theorem.

\begin{thm}\label{main}
Let $A=\{v_1,\dots,v_{d+2}\}\subset \mathbb{Z}^d$ be a set of
$d+2$ points, no $d+1$ of which are contained in the same
hyperplane. Then

        $$|hA|={h+d+1\choose d+1},\mbox{ for }1\leqslant h< {\rm Vol}(\Delta_A)d!/D$$
        and
        $$|hA|={h+d+1\choose d+1}-{h-{\rm Vol}(\Delta_A)d!/D+d+1\choose d+1},\mbox{ for }h\geqslant {\rm Vol}(\Delta_A)d!/D.$$
    \end{thm}

\medskip\noindent
{\bf Proof:}

Let $A=\{v_1,\dots,v_{d+2}\}\subset \mathbb{Z}^d$, and let $h$ be
a positive integer. First we will look for non-trivial solutions
of the system of equations in variables
$\alpha_1,\dots,\alpha_{d+2}$

$$\alpha_1v_1+\dots+\alpha_{d+2}v_{d+2}=\bfz,$$

$$\alpha_1+\dots+\alpha_{d+2}=0.$$
This is equivalent to

$$\alpha_1\widetilde{v}_1+\dots+\alpha_{d+2}\widetilde{v}_{d+2}=(\bfz,0).$$

Points $v_1,\dots,v_{d+2}$ are affine-dependent, so there exists a non-trivial solution $\mu_1,\dots,\mu_{d+2}$ of the above system of equations.

By the generic position, $D_{d+2}=\det (\widetilde{v}_1,\dots ,
\widetilde{v}_{d+1})\neq 0$, and also $\mu_{d+2}\neq 0$, since
otherwise points $v_1,\dots,v_{d+1}$ would be affine-dependent.

Multiplying the equality $\mu_1\widetilde{v}_1+\dots+\mu_{d+2}\widetilde{v}_{d+2}=(\bfz,0)$ by $1/\mu_{d+2}$, we get the identity

\begin{equation}\label{fla:mu-ovi}\frac{\mu_1}{\mu_{d+2}}\widetilde{v}_1+\dots+\frac{\mu_{d+1}}{\mu_{d+2}}\widetilde{v}_{d+1}=-\widetilde{v}_{d+2}.\end{equation}

By Cramer's rule, we have that

$$\frac{\mu_i}{\mu_{d+2}}=\det(\widetilde{v}_1,\dots,\widetilde{v}_{i-1},-\widetilde{v}_{d+2},
\widetilde{v}_{i+1},\dots,\widetilde{v}_{d+1})/\det(\widetilde{v}_1,\dots,\\\widetilde{v}_{d+1}).$$

Let
$\lambda_i:=\det(\widetilde{v}_1,\dots,\widetilde{v}_{d+1})\cdot
\frac{\mu_i}{\mu_{d+2}}=\det(\widetilde{v}_1,\dots,\widetilde{v}_{i-1},-\widetilde{v}_{d+2},
\widetilde{v}_{i+1},\dots,\widetilde{v}_{d+1})\in \mathbb{Z}$ for\\
$1\leqslant i\leqslant d+2$. Notice that $\lambda_i=\pm D_i$, i.e.
these numbers are equal up to the sign.

Multiplying identity (\ref{fla:mu-ovi}) by
$\det(\widetilde{v}_1,\dots,\widetilde{v}_{d+1})$, we get

\begin{equation}\label{fla:nula_u_1_2}\lambda_1\widetilde{v}_1+\dots+\lambda_k\widetilde{v}_k+\lambda_{k+1}\widetilde{v}_{k+1}+\dots+\lambda_{d+2}\widetilde{v}_{d+2}=(\bfz,0).\end{equation}

Now, since the coefficients $\lambda_i$ are all divisible by $D$,
we could divide this identity by $D$ and obtain

\begin{equation}\label{new}\frac{\lambda_1}{D}\widetilde{v}_1+\dots+\frac{\lambda_k}{D}\widetilde{v}_k+\frac{\lambda_{k+1}}{D}\widetilde{v}_{k+1}+\dots+\frac{\lambda_{d+2}}{D}\widetilde{v}_{d+2}=(\bfz,0).\end{equation}

Notice that all coefficients are integers and that they have no
common divisor.

Without loss of generality, we may assume that $\lambda_1,\dots,\lambda_k\geqslant 0$ and $\lambda_{k+1},\dots,\\\lambda_{d+2}<0$.

Let us now denote $r=\lambda_1+\dots+\lambda_k$. On a side note, we can deduce from this equation that the convex hull of $X_1=\{v_1,\dots,v_k\}$ intersects with the convex hull of $X_2=\{v_{k+1},\dots, v_{d+2}\}$. In particular, if $k=1$, then one set of vertices has a $d$-simplex as a convex hull, and the other is a vertex contained in the mentioned simplex.

Now, let $w\in hA$ have two representations (with non-negative coefficients):

$$\alpha_1v_1+\dots+\alpha_{d+2}v_{d+2}=\beta_1v_1+\dots+\beta_{d+2}v_{d+2}.$$
Then their difference is $\bfz$. Furthermore, the sum of
coefficients is $\sum_{i=1}^{d+2}(\alpha_i-\beta_i)=0$. Therefore,
the difference has to be a multiple of the left-hand side of
(\ref{new}). To each element $w\in hA$ corresponds exactly one
non-negative representation
$\alpha_1v_1+\dots+\alpha_{d+2}v_{d+2}$ for which $\alpha_i<
\lambda_i/D$ for at least one $1\leqslant i\leqslant k$. Namely,
if $\alpha_i\geqslant \lambda_i/D$ for $1\leqslant i\leqslant k$,
we can reduce this representation to the also non-negative
representation
$(\alpha_1-\lambda_1/D)\widetilde{v}_1+\dots+(\alpha_{d+2}-\lambda_{d+2}/D)\widetilde{v}_{d+2}$.
To obtain other non-negative representations, we can only add a
multiple of (\ref{new}). Therefore, to obtain the number of
elements in $hA$, we need to take the number of all non-negative
representations for which $\sum_{i=0}^{d+2}\alpha_i=h$, and reduce
it by the number of non-negative representations for which
$\sum_{i=0}^{d+2}\alpha_i=h$ and $\alpha_i\geqslant \lambda_i/D$,
for all $1\leqslant i\leqslant k$. Therefore, if $r/D \leqslant
h$, we have that

$$|hA|={d+h+1\choose h}-{d+1+h-r/D\choose h-r/D}={d+h+1\choose d+1}-{d+1+h-r/D\choose d+1}.$$
Otherwise, we have that

$$|hA|={d+h+1\choose d+1}.$$

Let us now determine $r$. We will denote the $d$-simplex determined by the vertices $v_1,\dots,v_{i-1},v_{i+1},\dots,v_{d+2}$ by $\Delta_i$. Note that $\lambda_i=\pm {\rm vol}(\Delta_i)\cdot d!$.

By the extension of Radon's theorem (Theorem \ref{unique}), every
generic point in $\Delta_A$ (that is not contained in any $d-1$
dimensional face of these simplices) is contained in exactly two
simplices $\Delta_i$ and $\Delta_j$, and they are such that the
vertices $v_i$ and $v_j$ belong to different sets $X_1$ and $X_2$.

Namely, if $x\in \Delta_i$, let $l$ be a half-line starting from $v_i$ passing through $x$, and let the final point of intersection of this half-line with the boundary of $\Delta_i$ belong to the face $(v_1,\dots,v_{i-1},v_{i+1},\dots,v_{j-1},v_{j+1},\dots,v_{d+2})$. Then $\Delta_j$ is the unique other simplex containing the point $x$. Since the interiors of $\Delta_i$ and $\Delta_j$ intersect, vertices $v_i$ and $v_j$ belong to the same half-space determined by the hyperplane spanned by the remaining $d$ vertices. By Radon's theorem, this means that one of $v_i$ and $v_j$ belongs to $X_1$, and the other belongs to $X_2$.

From this we see that $\Delta_A$ has a covering:

$$\Delta_A=\Delta_1\cup\dots\cup\Delta_k=\Delta_{k+1}\cup\dots\cup\Delta_{d+2}.$$

Since intersections of any two of the simplices $\Delta_1,\ldots
,\Delta_k $ and any two of the simplices $\Delta_{k+1},\ldots
,\Delta_{d+2} $ have volume $0$, and since
$\lambda_1,\dots,\lambda_k\geqslant 0$, we have that

$$r=\lambda_1+\dots+\lambda_k={\rm vol}(\Delta_1)\cdot d!+\dots+{\rm vol}(\Delta_k)\cdot d! ={\rm vol}(\Delta_A)\cdot d!.$$
\medskip

Now, we want to show that Theorem \ref{thm:1_2_u_CG} is a direct
corollary of Theorem \ref{main}. First, we prove the following

\begin{prop}
For a set $A=\{v_1,...,v_{d+2}\}$ for which the set $A-A$ generates
$\mathbb{Z}^d$ additively, the determinants $D_1,...,D_{d+2}$ have
no common divisor.
\end{prop}

\medskip\noindent
{\bf Proof:} Suppose, to the contrary, that the determinants
$D_1,...,D_{d+2}$ have common divisor $m\geq 2$. Notice that for
every $i\in \{1,2,...,d+1\}$, by subtracting the last column from
other columns we have

\begin{align*}
D_i&=\left\vert\begin{matrix}
v_1-v_{d+2} & v_2-v_{d+2} & \dots & v_{i-1}-v_{d+2} & v_{i+1}-v_{d+2} & \dots & v_{d+1}-v_{d+2} & v_{d+2}\\
0 & 0 & \dots & 0 & 0 & \dots & 0 & 1
\end{matrix}\right\vert\\
&=\left\vert\begin{matrix} v_1-v_{d+2} & v_2-v_{d+2} & \dots &
v_{i-1}-v_{d+2} & v_{i+1}-v_{d+2} & \dots & v_{d+1}-v_{d+2}
\end{matrix}\right\vert.\end{align*}

Since the vectors $v_1-v_{d+2},...,v_{d+1}-v_{d+2}$ generate
$\mathbb{Z}^d$ additively, then the unit vectors $e_j$ of standard
basis could be represented as the linear combinations with integer
coefficients of these vectors $v_i-v_{d+2}$. This implies that the
determinant of the identity matrix equals (by linearity) the
combination with integer coefficients of determinants, some of
which are $0$ (if they have two columns equal) and the remaining
are divisible by $m.$ This contradiction proves the proposition.
\bigskip

It is easy to see now that Theorem \ref{thm:1_2_u_CG} is a direct
corollary of Theorem \ref{main}. Namely, if we suppose that $A-A$
generates $\mathbb{Z}^d$ additively, by the above proposition,
$D=\mbox{GCD}(D_1,\dots,D_{d+2})=1$, and the Theorem
\ref{thm:1_2_u_CG} follows.
\bigskip

Let us now turn back to Example \ref{ex}. If we apply Theorem
\ref{main}, we see that $D_1,D_2,D_3,D_4=2$ and so $D=2$. So
theorem says

$$hA={h+3 \choose 3}- {h+3-2\cdot 2/2 \choose 3}={h+3\choose
3}- {h+1\choose 3}=(h+1)^2,$$

as we showed in Example \ref{ex}.

\section{Sumsets of a set with $d+3$ elements}

In this section we treat the case of the set $A$ of $d+3$ points
in $\mathbb{Z}^d$, especially the case when $d+1$ of them are the
vertices of a simplex containing the remaining two points.

Let us start with some examples illustrating the fact that this
case is more complicated. In particular, we will see that the
value $|hA|$ does not depend only on the convex hull $\Delta_A$ of
the set $A$ as in the previous case of the sets of $d+2$ points,
but also on the position of two remaining points inside
$\Delta_A$. Consequently, it is not a surprise that we do not
determine the exact value of $|hA|$, but provide the upper bound for this value.

\begin{example}

Let $d=1$. We will consider several sets of $4$ integers, all of
them containing integers $0,1$ and $8$ and the fourth integer
being one of $2,3,4,5,6,7$.

Let $A=\{0,1,2,8\}$. It is easy to see that for $h$ large enough,
the set $hA$ consists of all integers from $0$ to $8h$ (so, $8h+1$
of them) except for the integers $8h-1=8(h-1)+7, 8h-2=8(h-1)+6,
8h-3=8(h-1)+5, 8h-4=8(h-1)+4, 8h-5=8(h-1)+3, 8h-9=8(h-2)+7,
8h-10=8(h-2)+6, 8h-11=h(h-2)+5, 8h-17=8(h-3)+7\}$. Therefore,
$|hA|= 8h+1-9=8h-8$ in this case.

Similarly, if $h$ is large enough, for the set $A=\{0,1,3,8\}$ we
have $|hA|= 8h+1-7=8h-6$; for the set $A=\{0,1,4,8\}$ we have
$|hA|= 8h+1-9=8h-8$; for the set $A=\{0,1,5,8\}$ we have $|hA|=
8h+1-5=8h-4$; for the set $A=\{0,1,6,8\}$ we have $|hA|=
8h+1-3=8h-2$; and for the set $A=\{0,1,7,8\}$ we have $|hA|=
8h+1$.

Just for illustration, for $A=\{0,1,6,8\}$, the set $hA$ consists
of all integers from $0$ to $8h$ except for the integers
$8(h-1)+3, 8(h-1)+5,8(h-1)+7$. Notice that for the set
$A=\{0,1,7,8\}$, the set $hA$ consists of all integers from $0$ to
$8h$.

The convex hull of all these sets is the same, the interval
$[0,8]$, and all of them contain the same integer $1$. However,
the values of $|hA|$ differ.

\end{example}




Denote the vertices of $\Delta_A$ by $v_1,\dots ,v_{d+1}$, let $w$ be the $(d+2)^{\rm nd}$ element of $A$, and suppose that the $(d+3)^{\rm rd}$ is $\bfz$. Set

\begin{align*}
&\Lambda:={\rm span}_\mathbb{Z}(\widetilde{v}_1,\dots,\widetilde{v}_{d+1}),\\
&\Lambda^+:={\rm span}_\mathbb{N}(\widetilde{v}_1,\dots,\widetilde{v}_{d+1}),\\
&\Lambda_{(\bfz,1)}^+:={\rm span}_\mathbb{N}((\bfz,1),\widetilde{v}_1,\dots,\widetilde{v}_{d+1}),\\
&N_\Lambda:=\mbox{The number of integer points in the fundamental domain of }\Lambda,\\
\end{align*}

Let $\mathcal{C}_A$ be the cone over $A$. It is equal to

$$\bigcup_{m=0}^\infty \left((mw,m)+\Lambda_{(\bfz,1)}^+\right).$$
However, vector $(w,1)$ has finite order in the group $\mathbb{Z}^{d+1}/\Lambda$, which will be denoted by $o_w$. It can be seen that $o_w(w,1)\in \Lambda^+$.

To prove this, first note that $w$ belongs in the interior of simplex $\Delta_A$, which is why $(w,1)$ belongs to the boundary of the simplex determined by vertices $(\bfz,0),\widetilde{v}_1,\dots,\widetilde{v}_{d+1}$. Therefore, vector $(w,1)$ has barycentric coordinates $0\leqslant \mu_1,\dots\\ \dots ,\mu_{d+1}\leqslant 1$ such that

$$\sum_{i=1}^{d+1}\mu_i\widetilde{v}_i=(w,1).$$
By Lemma \ref{lem:racionalnost}, $\mu_i$ must be rational, and therefore $\mu_i=\frac{a_i}{q_i}$ for $1\leqslant i\leqslant d+1$, where $0\leqslant a_i\leqslant q_i$ and $(a_i,q_i)=1$. The order $o_w$ of $(w,1)$ is the least common container of $q_1,\dots,q_{d+1}$, ${\rm lcc}(q_1,\dots,q_{d+1})$. Since $o_w\frac{a_i}{q_i}$ are all non-negative integers, $o_w(w,1)\in \Lambda^+$. Therefore,

$$\mathcal{C}_A=\bigcup_{m=0}^{o_w-1}\left((mw,m)+\Lambda_{(\bfz,1)}^+\right).$$
This union need not be disjoint. From \cite[Theorem 1.2, simplicial case]{CG}, we have that $\Lambda_{(\bfz,1)}^+(t)=\frac{1-t^{{\rm Vol}(\Delta_A)d!}}{(1-t)^{d+2}}$. If $\mathcal{B}_A(t)$ is the generating series

\begin{align*}
\mathcal{B}_A(t)&=\sum_{m=0}^{o_w-1}t^m\Lambda_{(\bfz,1)}^+(t)=\sum_{m=0}^{o_w-1}t^m\left(1-t^{{\rm Vol}(\Delta_A)d!}\right)\sum_{h\geqslant 0}{h+d+1\choose h}t^h\\
&=\sum_{h\geqslant 0}b_ht^h,\end{align*}
the generating series $\mathcal{C}_A(t)=\sum_{h\geqslant 0}|hA|t^h$ will have coefficients $|hA|\leqslant b_h$ (If the union in the cone had been disjoint, there would have been an equality instead). From this we see an upper bound:

$$|hA|\leqslant$$

$$\leqslant \left\{\begin{array}{ll}
\sum\limits_{m=0}^h{m+d+1\choose m}, & h\leqslant o_w-1\\
\sum\limits_{m=0}^{o_w-1}{h+d+1-m\choose h-m}, & o_w\leqslant h\leqslant N_\Lambda-1\\
\sum\limits_{m=0}^{o_w-1}{h+d+1-m\choose h-m}-\sum\limits_{m=0}^{h-N_\Lambda}{m+d+1\choose m}, & N_\Lambda\leqslant h\leqslant N_\Lambda+o_w-1\\
\sum\limits_{m=0}^{o_w-1}{h+d+1-m\choose h-m}-\sum\limits_{m=0}^{o_w-1}{h-N_\Lambda+d+1-m\choose h-N_\Lambda-m}, & h\geqslant N_\Lambda+o_w.
\end{array}\right.$$

\bigskip

{\bf Acknowledgement:} I would like to thank my graduate thesis advisor Goran \DJ ankovi\' c for help.

\end{document}